\documentclass[11pt]{amsart}

\usepackage{amsmath,amssymb}
\usepackage{graphicx,hyperref}
\usepackage{color}
\usepackage{hyperref}
\usepackage{fancyhdr,float}

\renewcommand{\thispagestyle}[1]{} 

\theoremstyle{plain}
\newtheorem{thm}{Theorem}[section]
\newtheorem{observation}{Observation}[section]

\theoremstyle{definition}
\newtheorem{definition}{Definition}[section]

\theoremstyle{remark}
\newtheorem{claim}{Claim}[section]
\newtheorem{question}{Question}[section]

\begin{document}

\title{Knot mosaic tabulation}

\author[H. J. Lee]{Hwa Jeong Lee}
\address{Department of Mathematical Sciences, KAIST, 291 Daehak-ro, Yuseong-gu, Daejeon 34141, Korea}
\email{hjwith@kaist.ac.kr}
\thanks{The first author was supported by Basic Science Research Program through the National Research Foundation of Korea (NRF) funded by the Ministry of  Science, ICT $\&$ Future Planning (NRF-2015R1C1A2A01054607).}

\author[L. D. Ludwig]{Lewis D. Ludwig}
\address{Department of Math and Computer Science, Denison University, Granville, OH 43020, USA}
\email{ludwigl@denison.edu}

\author[J. S. Paat]{Joseph S. Paat}
\address{Department of Applied Mathematics and Statistics, Johns Hopkins, Baltimore, MD 21218-2682, USA}
\email{jpaat1@jhu.edu}

\author[A. Peiffer]{Amanda Peiffer}
\address{Department of Math and Computer Science, Denison University, Granville, OH 43020, USA}
\email{peiffe\_a1@denison.edu}


\begin{abstract}

In 2008, Lomonaco and Kauffman introduced a knot mosaic system to define a quantum knot system. A quantum knot is used to describe a physical quantum system such as the topology or status of vortexing that occurs on a small scale can not see. Kuriya and Shehab proved that knot mosaic type is a complete invariant of tame knots. In this article, we consider the mosaic number of a knot which is a natural and fundamental knot invariant defined in the knot mosaic system. We determine the mosaic number for all eight-crossing or fewer prime knots.  This work is written at an introductory level to encourage other undergraduates to understand and explore this topic.  No prior of knot theory is assumed or required.
\end{abstract}

\maketitle

\section{Introduction}
In this work we will determine the mosaic number of all 36 prime knots of eight crossings or fewer.  Before we do this, we will give a short introduction to knot mosaics.  Take a length of rope, tie a knot in it, glue the ends of the rope together and you have a {\it mathematical knot} -- a closed loop in 3-space.  A rope with its ends glued together without any knots is referred to as the {\it trivial knot}, or just an unknotted circle in 3-space.  There are other ways to create mathematical knots aside from rope.  For example, {\it stick knots} are created by gluing sticks end to end until a knot is formed (see Adams~\cite{A2004}).  In 2008, Lomonaco and Kauffman~\cite{LK2008} developed an additional structure for considering knots which they called {\it knot mosaics}.  In 2014, Kuriya and Shehab~\cite{KS2014} showed that this representation of knots was equivalent to tame knot theory, or knots with rope, implying that tame knots can be represented equivalently with knot mosaics.  This means any knot that can be made with rope can be represented equivalently with a knot mosaic.\\

A {\it knot mosaic} is the representation of a knot on an $n\times n$ grid composed of 11 tiles as depicted in Figure~\ref{F:Tiles}.  A tile is said to be {\it suitably connected} if each of its connection points touches a connection point of a contiguous tile.  Several examples of knot mosaics are depicted in Figure~\ref{F:Example}.  It should be noted that in Figure~\ref{F:Example}, the first mosaic is a knot, the trefoil knot, the second mosaic is a link, the Hopf Link, and the third is the composition of two trefoil knots (remove a small arc from two trefoils then connect the four endpoints by two new arcs depicted in red, denoted by $3_1 \# 3_1$).  A knot is made of one component (i.e. one piece of rope), and a link is made of one or more components (i.e. one or more pieces of rope).  For this work, we will focus on knot mosaics of {\it prime knots}.  A prime knot is a knot that cannot be depicted as the composition of two non-trivial knots.  The trefoil is a prime knot. 

\begin{figure}[htp]
\begin{center}
\includegraphics[scale=0.92]{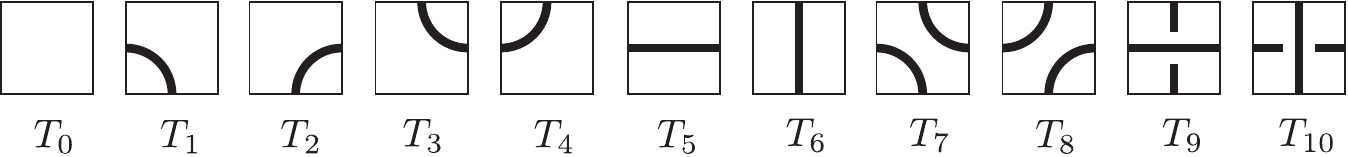}
\vspace{-2mm}
\caption{Tiles $T_0$\,--\,$T_{10}$ respectively.}\label{F:Tiles}
\end{center}
\end{figure}

\begin{figure}[htb]
   \begin{minipage}{1.25in}
   \begin{center}
    \includegraphics[height=0.9in,clip]{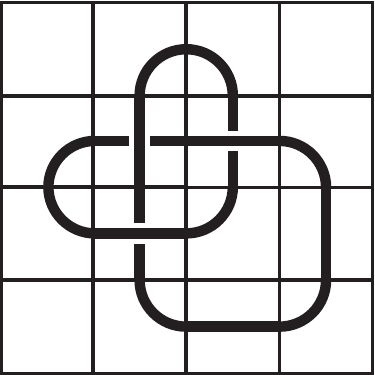}\\ \vspace{1.5mm}
   {Trefoil}
   \end{center}
    \end{minipage}%
    \/ \/
    \begin{minipage}{1.25in}
   \begin{center}
   \includegraphics[height=0.9in,clip]{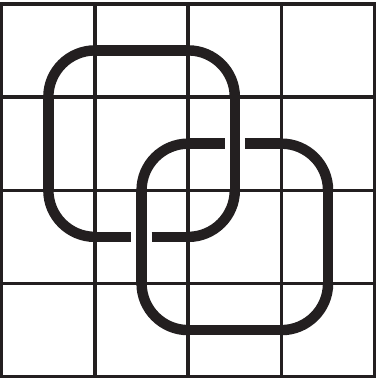}\\ \vspace{1.5mm}
   {Hopf Link}
   \end{center}
    \end{minipage}
     \begin{minipage}{1.25in}
   \begin{center}
   \includegraphics[height=0.9in,clip]{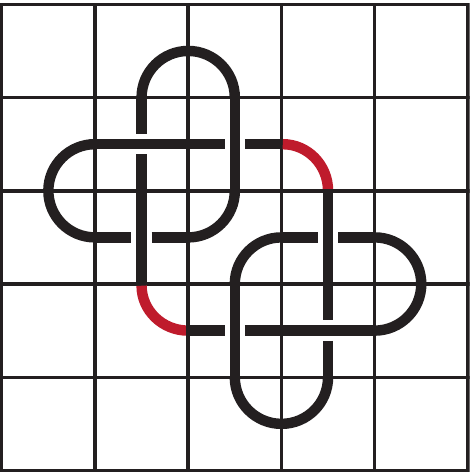}\\ \vspace{1.5mm}
   {Composition}
   \end{center}
    \end{minipage}
    \begin{minipage}{1.35in}
   \begin{center}
   \includegraphics[height=0.9in,clip]{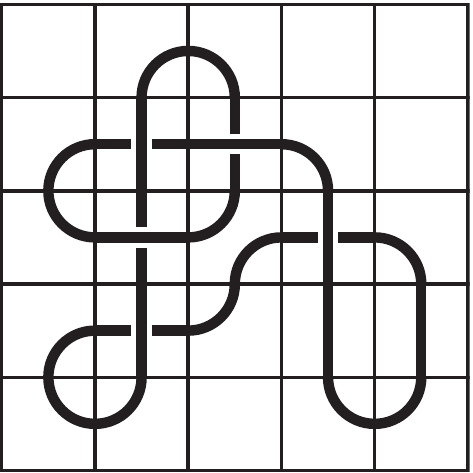} \\ \vspace{1.5mm}  
   {Non-reduced Trefoil}
   \end{center}
    \end{minipage}%
    \caption{Examples of link mosaics.}
    \label{F:Example}
\end{figure}

\medskip

When studying knots, a useful and interesting topic used to help distinguish two knots is {\it knot invariants}.  A knot invariant is a quantity defined for each knot that is not changed by {\it ambient isotopy}, or continuous distortion, without cutting or gluing the knot.  One such knot invariant is the {\it crossing number} of a knot.  The crossing number is the fewest number of crossings in any diagram of the knot.  For example, the crossing number of the trefoil is three, which can be seen in Figure~\ref{F:Example}.  A {\it reduced} diagram of a knot is a projection of the knot in which none of the crossings can be reduced or removed.  The fourth knot mosaic depicted in Figure~\ref{F:Example} is an example of a non-reduced trefoil knot diagram.  In this example the crossing number of three is not realized because there are two extra crossings that can be easily removed. \\

An interesting knot invariant for knot mosaics is the {\it mosaic number}.  The mosaic number of a knot $K$ is the smallest integer $n$ for which $K$ can be represented on an $n\times n$ mosaic board.  We will denote the mosaic number of a knot $K$ as $m(K)$.  For the trefoil, it is an easy exercise to show that mosaic number for the trefoil is four, or $m({\rm 3_1})=4$.  To see this, try making the trefoil on a $3\times 3$ board and arrive at a contradiction.\\

Next, we introduce a technique that can be used to ``clean up" a knot mosaic by removing unneeded crossing tiles.  In 1926, Kurt Reidemeister demonstrated that two knot diagrams belonging to the same knot, up to ambient isotopy, can be related by a sequence of three moves, now know as the Reidemeister moves~\cite{R1926}.  For our purposes, we will consider two of these moves on knot mosaics, the mosaic Reidemeister Type I and Type II moves as described by Lomonaco and Kauffman~\cite{LK2008}.  For more about Reidemeister moves, the interested reader should see Adams~\cite{A2004}.\\


The mosaic Reidemeister Type I moves are the following:
\begin{figure}[ht]
\begin{minipage}{1.4in}
\includegraphics[height=.5in,clip]{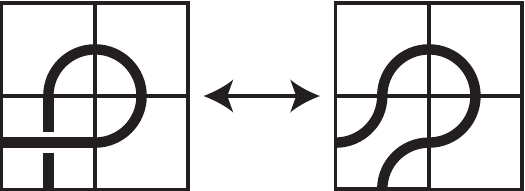}
\end{minipage}
\mbox{}\hspace{.4in}
\begin{minipage}{1.4in}
\includegraphics[height=.5in,clip]{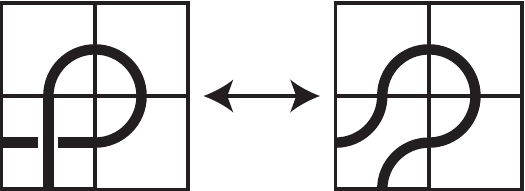}
\end{minipage}
\caption{Reidemeister Type I moves.}\label{F:R1}
\end{figure}
\medskip

The mosaic Reidemeister Type II moves are given below:
\begin{figure}[htp]
\begin{minipage}{1.4in}
\includegraphics[height=.5in,clip]{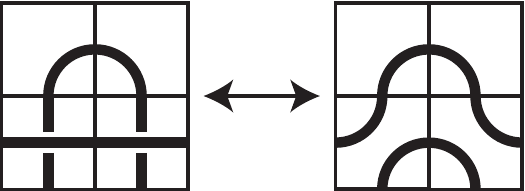}
\end{minipage}
\mbox{}\hspace{.4in}
\begin{minipage}{1.4in}
\includegraphics[height=.5in,clip]{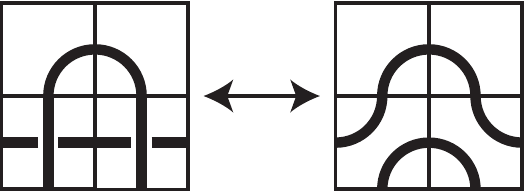}
\end{minipage}
\mbox{}\hspace{.4in}
\begin{minipage}{1.4in}
\includegraphics[height=.5in,clip]{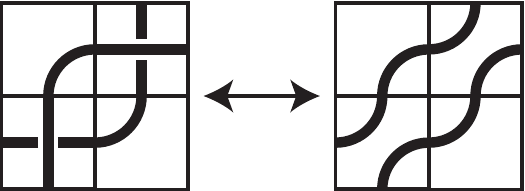}
\end{minipage}
\caption{Reidemeister Type II moves.}\label{F:R2}
\end{figure}

\medskip

Next we make several observations that will prove useful.

\medskip

\begin{observation}[The two-fold rule]\label{O:TwoFold}
Once the inner tiles of a mosaic board are suitably placed, there are only two ways to complete the board so that it is suitable connected, resulting in a knot or a link~\cite{H2014}.
\end{observation}

\medskip

For any given $n\times n$ mosaic board, we will refer to the collection of inner tiles as the {\it inner board}.  For example, in Figure~\ref{F:5by5} the tiles $I_1$-$I_9$ would make the inner board for this mosaic board.

\medskip

\begin{observation}\label{O:EvenAll}
Assume $n$ is even. For a board of size $n$ with the inner board consisting of all crossing tiles, any resulting suitably connected mosaic will either be a $n-2$ component link mosaic or $n-3$ component non-reduced link mosaic, see Figure~\ref{F:4_1} for example.
\end{observation}

\medskip

\begin{observation}\label{O:OddAll}
Assume $n$ is odd. For a board of size $n$ with the inner board consisting of all crossing tiles, any resulting suitably connected mosaic is a non-reduced knot mosaic.
\end{observation}

\medskip

Observation~\ref{O:OddAll} can be generalized in the following way.

\medskip

\begin{observation}\label{O:Corner} Let $M$ be a knot mosaic with two corner crossing tiles in a top row of the inner board. If the top boundary of the row has an odd number of connection points, then a Reidemeister Type I move can be applied to either corner of the row of $M$.  This extends via rotation to the outer-most columns and the lower-most row of the inner board. 
\end{observation}
\begin{figure}[ht]
   \begin{center}
      \includegraphics[scale=1.05]{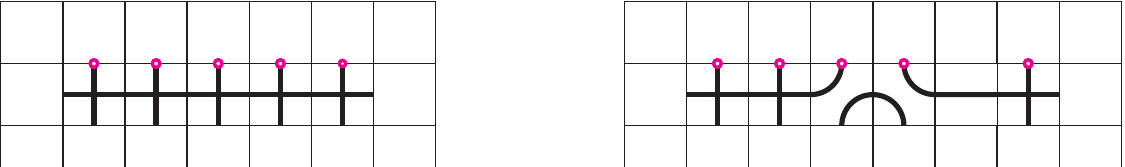}
      \label{fig:5by5}
      \caption{Two examples of knot mosaics with an odd number of connection points on the top boundary of the top row of the inner board; a connection point on the boundary is marked by red circle.}
   \end{center}
\end{figure}

\medskip

Armed with this quick introduction to knot mosaics, we are ready to determine the mosaic number for all prime knots with a crossing number of eight or fewer in the next section.  Before we proceed, it should be noted that there are many other questions to consider regarding knot mosaics besides finding the mosaic number.  For example, what is the fewest number of non-blank tiles needed to create a specific knot?  This could be known as the {\it tile number} of a knot.  With this in mind, if we allow knot mosaics to be rectangular, can some knots have a smaller tile number if they are presented in a rectangular, $m\times n$ configuration as opposed to a square configuration?  We will conclude this article with a number of other open questions about knot mosaics that you can consider and try to solve.\\

\section{Determining the mosaic number of small prime knots}
In this section, we will determine the mosaic number for all prime knots of eight or fewer crossings.  We are referring to these as ``small" prime knots.  We will see that for some knots, the mosaic number is ``obvious," while others take some considerable work.  We begin with knots\footnote{At this point we adopt the Alexander-Briggs notation for knots.  The number represents the number of crossings, while the subscript represents the order in the table as developed by Alexander-Briggs and extended by Rolfsen.  The knot $6_2$ is the second knot of six crossings in the Rolfsen knot table ~\cite{A2004}.} of an obvious mosaic number as shown in Table~\ref{T:Obvious}.\\

Why do these knots have obvious mosaic numbers?  As previously noted the trefoil knot, $3_1$, cannot fit on a $3 \times 3$ mosaic board, as such a board would only allow one crossing tile when $3_1$ requires at least three crossings.  Hence, the mosaic number for the trefoil is obvious.  Similarly, the knots $5_1$, $5_2$, and $6_2$ have more than four crossings, so they cannot fit on a $4\times 4$ mosaic board which only allows at most four crossing tiles.  In the Appendix, we have provided representations of these knots on $5\times 5$ mosaic boards, thus determining the mosaic number for these knots.  It should be noted that as the knots become larger, it is often difficult to determine whether a specific knot mosaic represents a given knot.  To check that a knot mosaic represents a specific knot, we used a software packaged called KnotScape~\cite{K1999} developed by Professor Morwen Thistlethawite which looks at the Dowker notation of a knot to determine the knot presented.  While KnotScape cannot determine all knots, it can determine small prime knots.  For more see Adams~\cite{A2004}.\\

Next we consider knots whose mosaic number is ``almost obvious."  At first glance, one may think that the figure-eight knot, $4_1$, should have a mosaic number of four.  Start with a $4\times 4$ mosaic with the four inner tiles being crossing tiles.  By Observation~\ref{O:TwoFold}, this $4\times 4$ mosaic can be completed in two ways as seen in Figure~\ref{F:4_1}.  However, the knot $4_1$ is known as an {\it alternating} knot.  An alternating knot is a knot with a projection that has crossings that alternate between over and under as one traverses around the knot in a fixed direction.  So, if we were to try to place $4_1$ on a $4\times 4$ mosaic, there would be four crossing tiles and they would have to alternate.  Thus $4_1$ cannot be placed on a $4 \times 4$ board.  In the Appendix we see a presentation of $4_1$ on a $5\times 5$ mosaic board, hence $m(4_1)=5$.

\begin{figure}[htp]
\begin{center}
\begin{minipage}{1.2in}
\begin{center}
\includegraphics[height=1in,clip]{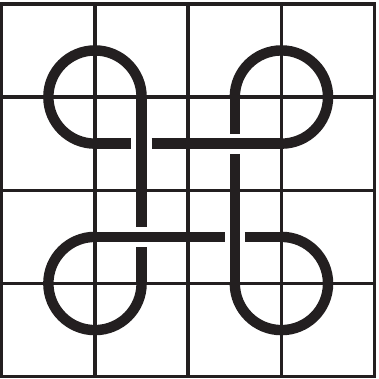}\\  \vspace{1.5mm}
Unknot
\end{center}
\end{minipage}
\mbox{}\hspace{.3in}
\begin{minipage}{1.8in}
\begin{center}
\includegraphics[height=1in,clip]{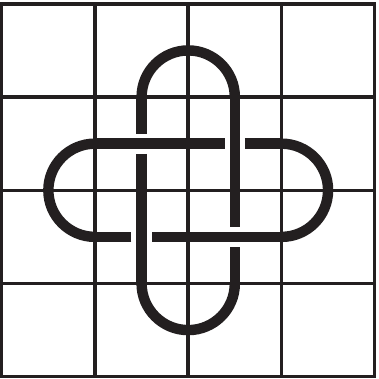}\\  \vspace{1.5mm}
King Solomon knot (link)
\end{center}
\end{minipage}
\caption{Possible four crossing mosaics on $4\times 4$ board.}\label{F:4_1}
\end{center}
\end{figure}

Another knot with an almost obvious mosaic number is $6_1$.  Figure~\ref{F:Compare6_1} first depicts a configuration of $6_1$ on a $6\times 6$ mosaic board.  However, by performing a move called a {\it flype} (see Adams~\cite{A2004}) we can fit $6_1$ on a  $5\times 5$ mosaic board.  It should be noted that this mosaic representation of $6_1$ has seven crossings instead of six.  Thus the mosaic number for $6_1$ is realized when the crossing is not.  It turns out that $6_1$ is not the only knot with such a property.  Ludwig, Evans, and Paat~\cite{L2013} created an infinite family of knots whose mosaic numbers were realized only when their crossing numbers were not.

\begin{figure}[htp]
\begin{center}
\begin{minipage}{1.3in}
\includegraphics[height=1.1in,clip]{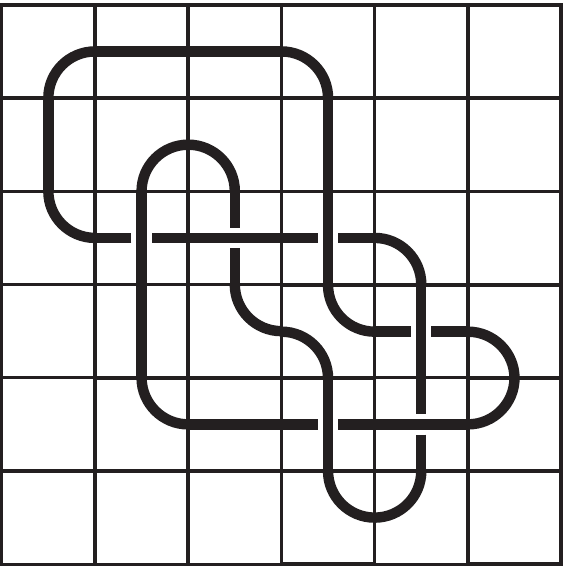}
\end{minipage}
\mbox{}\hspace{.6in}
\begin{minipage}{1.3in}
\includegraphics[height=1.1in,clip]{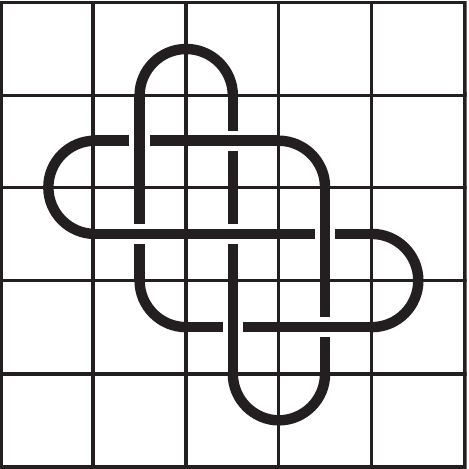}
\end{minipage}
\end{center}
\caption{The knot $6_1$ as a 6-mosaic and a 5-mosaic.}\label{F:Compare6_1}
\end{figure}

\medskip

At this point, we have determined the mosaic number for all six or fewer crossing knots except $6_3$.  Surprisingly, we will see that $6_3$ cannot fit on a $5\times 5$ board, even though such a board has nine possible positions to place crossing tiles.  \\

\medskip

\begin{thm}\label{T:6-3}
The mosaic number of the knot $6_3$ is six; that is $m(6_3)=6$.
\end{thm}
\begin{proof}
Since $6_3$ has six crossings, we know that $m(6_3)\geq 5$.  In the Appendix we see a representation of $6_3$ on a $6\times 6$ mosaic board, so $m(6_3)\leq 6$.  This mean $m(6_3)=5$ or $m(6_3)=6$.  We now argue $m(6_3)=6$. \\ 

Assume to the contrary that $m(6_3)= 5$. By the definition of the mosaic number, this implies that there is some $5\times 5$ mosaic $M$ that represents $6_3$. We will show via a case analysis that regardless of how the crossing tiles are arranged on $M$, the resulting knot is in fact \textit{not $6_3$}. This will give us a contradiction, implying that $m(6_3)=6$. \\

In order to help with the case analysis, we label the nine inner tiles of the $5\times 5$ mosaic $I^1$ -- $I^9$, as depicted in Figure~\ref{F:5by5}.  The $6_3$ knot uses at least six crossing tiles.  Since $6_3$ has a crossing number of six, by the Pigeon Hole Principle at least one of the four corner inner tiles, $I^1$, $I^3$, $I^7$, or $I^9$ must be a crossing tile.  
By rotations, it is enough to consider the four cases that are depicted in Figure~\ref{F:FourCases}. Note that in Figure~\ref{F:FourCases}, gray tiles describe non-crossing tiles and white tiles could be crossing tiles (but they do not have to be).
\begin{figure}[ht]
   \begin{center}
      \includegraphics[scale=0.9]{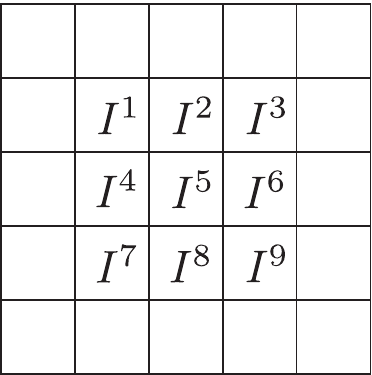}
      \caption{The $5\times 5$ mosaic board with inner-tiles $I^1$ -- $I^9$.}
   \label{F:5by5}
   \end{center}
\end{figure}

\begin{figure}[htbp]
\begin{center}
\includegraphics[width=4.5in]{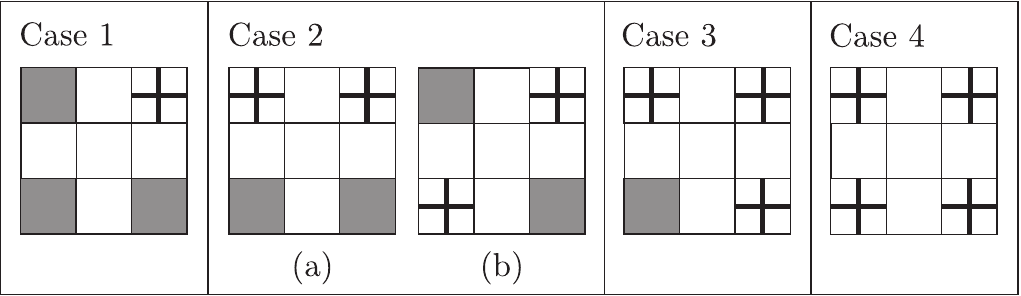}
\caption{Four different cases for placing crossing tiles on the inner corner tiles for a $5\times 5$ mosaic.}
\label{F:FourCases}
\end{center}
\end{figure}

\noindent {\bf Case 1}: Suppose that  $I^3$ is a crossing tile, while $I^1, I^7$, and $I^9$ are not. Thus $I^2, I^4, I^5, I^6,$ and $I^8$ are all crossing tiles since $M$ has exactly 6 crossing tiles. Every reduced projection of an alternating knot is alternating, and since $M$ represents the alternating knot $6_3$, the crossings on $M$ must alternate.  A quick inspection shows that to suitably connect the inner tiles of $M$ we would need to ensure that (i) $I^1, I^7$, and $I^9$ are not crossing tiles, (ii) the crossings are alternating, and (iii) there are no easily removed crossings, however this results in a mosaic that represents $6_2$, as seen in Figure~\ref{Fig:Ex1}.  Therefore, $6_3$ cannot be constructed in Case 1.

\begin{figure}[htbp]
\begin{center}
\includegraphics[width=1.2in]{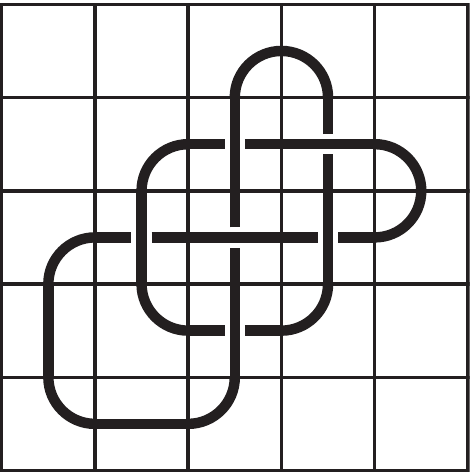}
\caption{$6_2$ results from Case 1.}
\label{Fig:Ex1}
\end{center} 
\end{figure}


\noindent  {\bf Case 2}: For the case when $M$ has two corner inner tiles that are crossings, we require two sub-cases.\\

\textbf{Sub-case 2(a)}: Suppose that  $I^1$ and $I^3$ are crossing tiles, while $I^7$ and $I^9$ are not. If $I^2$ is a crossing tile then Observation~\ref{O:Corner} may be applied to the top inner row of $M$. Applying this observaiton will either change $I^1$ or $I^3$ to a non-crossing tile.  Without loss of generalization, suppose that $I^1$ is changed. Notice that $M$ now satisfies Case 1, and from the previous analysis, $M$ does not represent $6_3$. Therefore we may assume that $I^2$ is not a crossing tile. 
Since $M$ has at least 6 crossing tiles, and $I^2, I^7,$ and $I^9$ are not crossing tiles, the remaining 6 inner tiles must be crossing tiles.  Then $I^2$ has $4$ connection points. By Observation~\ref{O:Corner} either $I^1$ or $I^3$ can be changed to a non-crossing tile. Again $M$ falls into Case 1 and does not represent $6_3$. \\

{\bf Sub-case 2(b)}   Suppose that  $I^3$ and $I^7$ are crossing tiles while $I^1$ and $I^9$ are not. \\

\begin{claim}
In Sub-case 2(b), if $M$ has 6 crossings tiles then $M$ cannot be $6_3$.
\end{claim}
\begin{proof}[Proof of Claim]
Assume to the contrary that $M$ only has $6$ crossing tiles.  Then exactly one of $\{I^2, I^4, I^5, I^6, I^8\}$ is a non-crossing tile.  Up to rotation and reflection, we only need to consider situations (i) and (ii), where
\begin{itemize}
\item[(i)] $I^3, I^4, I^5, I^6, I^7$ and $I^8$ are the only crossing tiles,
\item[(ii)] $I^2, I^3, I^4, I^6, I^7$ and $I^8$ are the only crossing tiles.
\end{itemize}
There are two crossing arrangements up to mirror to describe case (ii) as shown in the last two images in Figure~\ref{F:Case2(b)inner}.


\begin{figure}[ht]
   \begin{center}
      \includegraphics[scale=.6]{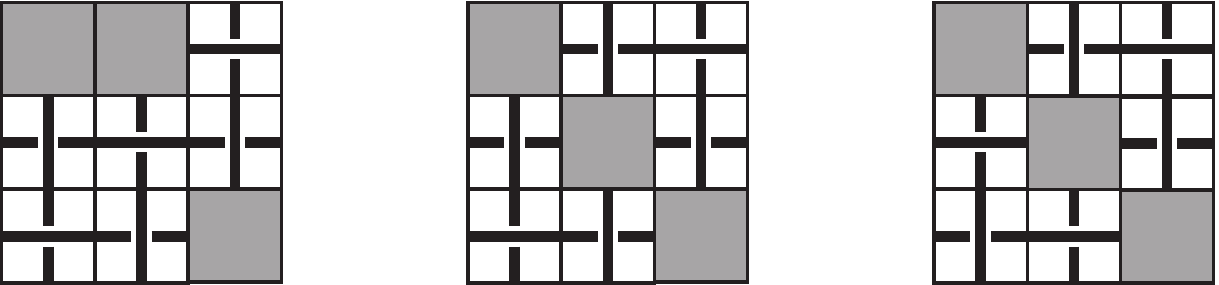}
      \caption{Possible configurations of six crossing tiles under Sub-case 2(b)}
      \label{F:Case2(b)inner}
   \end{center}
\end{figure}

\medskip

Since $M$ has 6 crossings, the corner inner-tiles $I^3$ and $I^7$ cannot be changed from crossing tiles. With this in mind, suitably connecting the crossing tiles in Figure~\ref{F:Case2(b)inner} so that a knot (and not a 2-component link) is created leads to $6_2$ or $3_1\#3_1$ (see Figure~\ref{F:Case2(b)knots}).  This is a contradiction, proving our claim.

\begin{figure}[ht]
   \begin{center}
      \includegraphics[scale=.65]{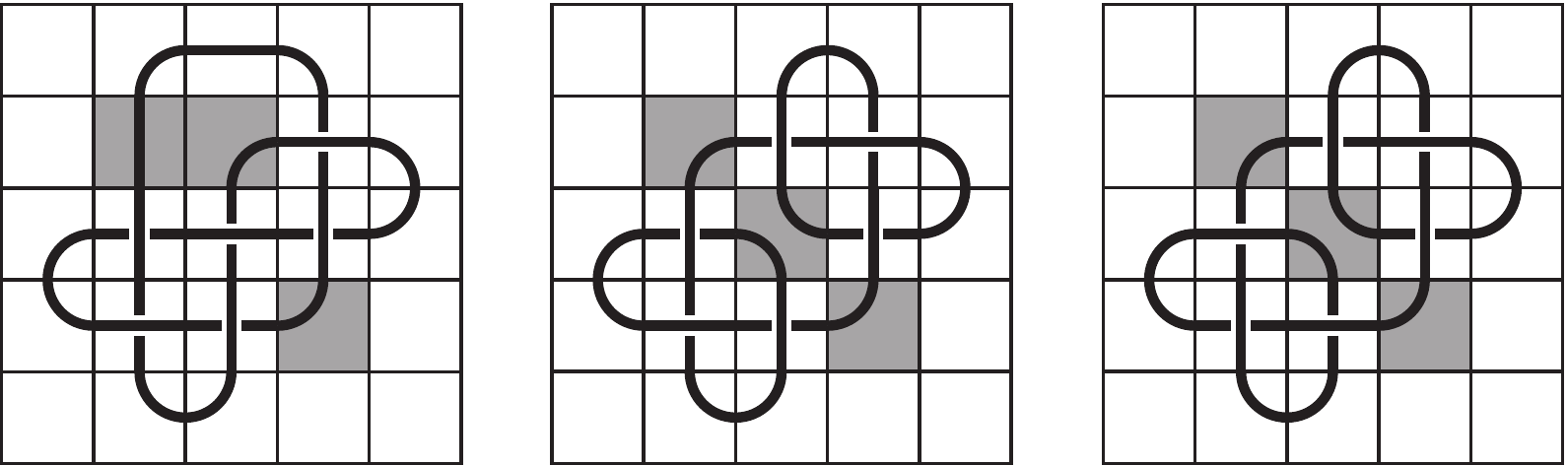}
      \caption{Knot mosaic configurations of $6_2$ and $3_1\#3_1$.}
      \label{F:Case2(b)knots}
   \end{center}
\end{figure}

\end{proof}

\begin{claim}
In Sub-case 2(b), if $M$ has 7 crossings tiles then $M$ cannot be $6_3$.
\end{claim}

\begin{proof}[Proof of Claim]
If the crossing tiles are alternating, then $M$ represents $7_4$, contradicting that $M$ is $6_3$.  So assume that $M$ has 7 crossings and is non-alternating.\\

Observe that if any of the pairs $\{I^2, I^3\}$, $\{I^3, I^6\}$, $\{I^4, I^7\}$, or $\{I^7, I^8\}$ are non-alternating, then a Type II Reidemeister move is present, and $M$ can be reduced to five crossings. However, this contradicts that $M$ is $6_3$. So each pair $\{I^2, I^3\}$, $\{I^3, I^5\}$, $\{I^4, I^7\}$, and $\{I^7, I^8\}$ is alternating. Since $M$ is non-alternating, at least one of the pairs $\{I^2, I^5\}$, $\{I^6, I^5\}$, $\{I^4, I^5\}$, or $\{I^5, I^8\}$ is non-alternating. Without loss of generality, assume $\{I^2, I^5\}$ creates a pair of non-alternating crossings. Then, up to ambient isotopy, $M$ is $6_1$ or $3_1$, as seen in Figure~\ref{fig_6131}. Therefore Sub-case 2(b) does not result in $6_3$.

\begin{figure}[ht]
   \begin{center}
      \includegraphics[scale=.65]{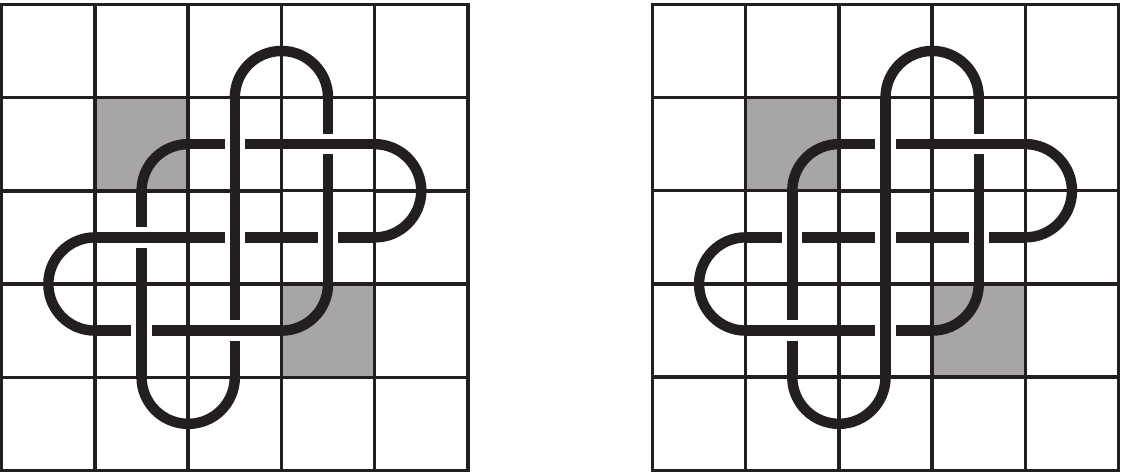}
      \vspace{3mm}
      \caption{$6_1$ and $3_1$, respectively.}
      \label{fig_6131}
   \end{center}
\end{figure}

\vspace{-8mm}

\end{proof}


\noindent\textbf{Case 3:} Suppose that  $I^1,I^3$ and $I^9$ are crossing tiles, and $I^7$ is not. Since $M$ has at least 6 crossings, there must be at least 3 more crossing tiles on the board. \\

Observe that if either $I^2$ or $I^6$ are crossing tiles then Observation~\ref{O:Corner} may be applied to the top inner row or the right inner column, respectively. As a result of this, one of the crossings on $I^1,I^3$ and $I^9$ can be changed to a non-crossing tile, leaving only two corner inner-tiles that are crossings. This reverts to Case 2 showing that $M$ would not represent $6_3$. Therefore, we may assume that neither $I^2$ nor $I^6$ are crossing tiles.\\

With $I^2$ and $I^6$ eliminated as crossing tiles, $I^1, I^3, I^4, I^5, I^8,$ and $I^9$ must all be crossing tiles. Then by Observation~\ref{O:Corner}, either $I^3$ or $I^9$ can be changed from a crossing tile to a non-crossing tile. This again reverts to Case 2 showing that $M$ would not represent $6_3$.\\

\noindent\textbf{Case 4:} Suppose that  $I^1, I^3, I^7$ and $I^9$ are crossing tiles. Note that at least one of the tiles in the set $\{I^2, I^4,I^6,I^8\}$ must be a crossing tile. This means Observation~\ref{O:Corner} applies to some row or column of $M$, and $M$ can be reduced to Case 3. Hence $M$ does not represent $6_3$.\\

By the above four cases, we see that the $6_3$ cannot be placed on a $5\times 5$ mosaic. Hence, by the figure for $6_3$ in the Appendix, we see that the mosaic number of $6_3$ is six; that is $m(6_3)=6$.

\end{proof}

\medskip

We next consider the seven-crossing knots.  As seen above, the knot $7_4$ can be placed on a $5\times 5$ mosaic board.  We formalize this result in the following proposition as well as establish the mosaic number for the other seven-crossing knots.\\

\begin{thm}\label{T:7-4}
The mosaic number of $7_4$ is five, that is $m(7_4)=5$.  Moreover, $7_4$ is the only seven-crossing prime knot with mosaic number five; the remaining seven-crossing prime knots have mosaic number six.
\end{thm}
\begin{proof}
We have already seen via Observation~\ref{O:Corner} that at most seven crossing tiles can be placed on a $5 \times 5$ board without reduction via a Reidemeister Type I move.  Moreover, since all seven crossing knots are alternating, there is only one way to place seven alternating crossing tiles up to mirror, reflection, and rotation as depicted in Figure~\ref{F:Alt_7}.  When this arrangement is suitably connected, the only knot resulting is $7_4$.  Therefore $m(7_4)=5$ and all other seven-crossing knots have mosaic number six as depicted in the Appendix.

\begin{figure}[ht]
   \begin{center}
      \includegraphics[scale=.65]{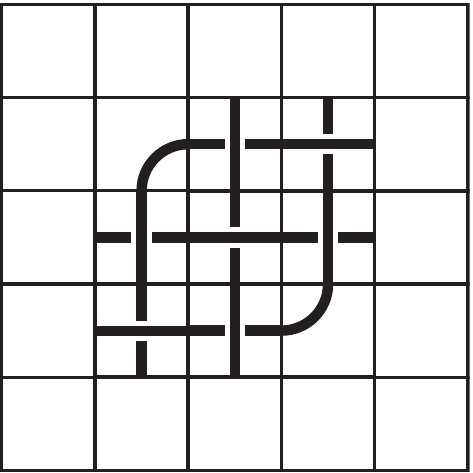}
      \vspace{3mm}
      \caption{Seven alternating tiles on a $5\times5$ board.}
      \label{F:Alt_7}
   \end{center}
\end{figure}
\end{proof}

Next we consider the eight-crossing knots.  Let $K$ be a knot of eight crossings. By the proof of Theorem~\ref{T:6-3}, we know that $K$ cannot fit on a $5\times 5$ mosaic board. This means $m(K)\geq 6$. Furthermore there exists a knot mosaic of $K$ on a $6\times 6$ board (see the Appendix). This implies $m(K)=6$.\\

Given the above arguments, we see that the mosaic number of the eight crossing knots is greater than five.  By the Appendix and use of KnotScape, we see that the mosaic number of all eight crossing knots is six.  We summarize our findings in the Table~\ref{T:Obvious}. For each knot $K$ with at most 8 crossings, the Appendix includes a mosaic of size $m(K)$ representing $K$.  \\

\begin{table}[htp]
\renewcommand\arraystretch{1.45}
\begin{center}
\begin{tabular}{|c|c||c|c||c|c||c|c||c|c||c|c|}\hline
$K$ & $m(K)$ & $K$ & $m(K)$ & $K$ & $m(K)$ & $K$ & $m(K)$ & $K$ & $m(K)$ & $K$ & $m(K)$\\ \hline
$0_1$ & $2^{\dag}$ & $6_2$ & $5^{\dag}$ & $7_5$ & $6^{\sharp}$ & $8_4$ & 6 & $8_{10}$ & 6 & $8_{16}$ & 6\\ \hline
$3_1$ & $4^{\dag}$ & $6_3$ & $6^{\natural}$ & $7_6$ & $6^{\sharp}$ & $8_5$ & 6 & $8_{11}$ & 6 & $8_{17}$ & 6\\ \hline
$4_1$ & $5^{\ddag}$ & $7_1$ & $6^{\sharp}$ & $7_7$ & $6^{\sharp}$ & $8_6$ & 6 & $8_{12}$ & 6 & $8_{18}$ & 6\\ \hline
$5_1$ & $5^{\dag}$ & $7_2$ & $6^{\sharp}$ & $8_1$ & 6  &  $8_7$ & 6 & $8_{13}$ & 6 & $8_{19}$ & 6\\ \hline
$5_2$ & $5^{\dag}$ & $7_3$ & $6^{\sharp}$ & $8_2$ & 6 & $8_8$ & 6 &  $8_{14}$ & 6 &  $8_{20}$ & 6\\ \hline 
$6_1$ & $5^{\dag}$ & $7_4$ & $5^{\sharp}$ & $8_3$ & 6 & $8_9$ & 6 &  $8_{15}$ & 6 & $8_{21}$ & 6\\ \hline 
\end{tabular}
\end{center}
\vspace{3mm}
\caption{Mosaic number of knots with up to $8$ crossing}
\label{T:Obvious}
\end{table}%

\vspace{-4mm}
\begin{itemize}
\item[$\dag$] Obvious. 
\item[$\ddag$] From Observation~\ref{O:EvenAll}, we have $m(K)\geq 5$.
\item[$\natural$] by Theorem~\ref{T:6-3}
\item[$\sharp$] by Theorem~\ref{T:7-4}
\end{itemize}

\medskip

\section{Further work}
We conclude with a number of open questions that would make good projects for undergraduate research.  To begin, often times in mathematics when something is proved for the first time, the proof may not be very elegant.  For example, Newton's ``proofs" of various facts in calculus look much different than the proofs you would find in typical calculus book today.  Although Theorem~\ref{T:6-3} proved $m(6_3)=6$, could the proof be shortened?

\begin{question}
Is there a more direct proof of $m(6_3)=6$?
\end{question}

It is often interesting to see how various knot invariants compare.  For example, in 2009, Ludwig, Paat, and Shapiro showed that the mosaic number and crossing number of a knot can be related in the following way:
$$\lceil \sqrt{c(k)} \rceil +3\leq m(k).$$
Moreover, in 2014 Lee et al.~\cite{L2014} showed 
$$m(K)\leq c(k)+1.$$

\medskip

\begin{question}
Do tighter upper or lower bounds exist on the mosaic number of a knot using the crossing number of the knot?
\end{question}

\medskip

In 2013, Ludwig, Evans, and Paat~\cite{L2013} created an infinite family of knots whose mosaic number is realized when the crossing number is not.  We have seen that $6_1$ is such a knot.  The mosaic number for $6_1$ is five, but in that projection, the number of crossing tiles is seven.  To realize the crossing number for $6_1$ it has to be projected on a $6\times 6$ mosaic board. In general, Ludwig et al. created a family of knots whose mosaic number was realized on an $n\times n$ mosaic board with $n$ odd, $n\geq 5$ and whose crossing number was realized on an $(n+1)\times (n+1)$ mosaic board.

\medskip

\begin{question}
Does there exist an infinite family of knots with mosaic number $n\times n$ whose crossing number is realized on an $(n+k)\times (n+k)$ for each $k\geq 2$?
\end{question}

\medskip

Working with undergraduates in 1997, Adams et al.~\cite{A1997} proved the surprising fact that the composition of $n$ trefoils has stick number exactly $2n+4$.  Is there a similar result for knot mosaics?

\medskip

\begin{question}
What is the mosaic number of the composition of $n$ trefoil knots?
\end{question}

\medskip

We have briefly touched on mosaic Reidemeister moves.  It is often the case that to make these moves, one has to add a certain number of rows and columns to the mosaic board to provide enough room for the moves to occur (see Kuriya and Shebab~\cite{KS2014} for example). Is such an expansion always necessary?

\medskip

\begin{question}
Let $M_1$ and $M_2$ be $n$-mosaics of that represent the same knot.  Is there a set of mosaic Reidemeister moves from $M_1$ to $M_2$ on a mosaic board of size $n\times n$?
\end{question}

\medskip

Another well-studied area of knot theory is braid diagrams.  In 1923 J.W. Alexander proved that every knot or link has a closed braid representation~\cite{A1923}.  From Figure~\ref{F:Braid}, we see that braids appear ``rectangular" in nature which leads to the next question.

\begin{figure}[ht]
   \begin{center}
      \includegraphics[scale=.5]{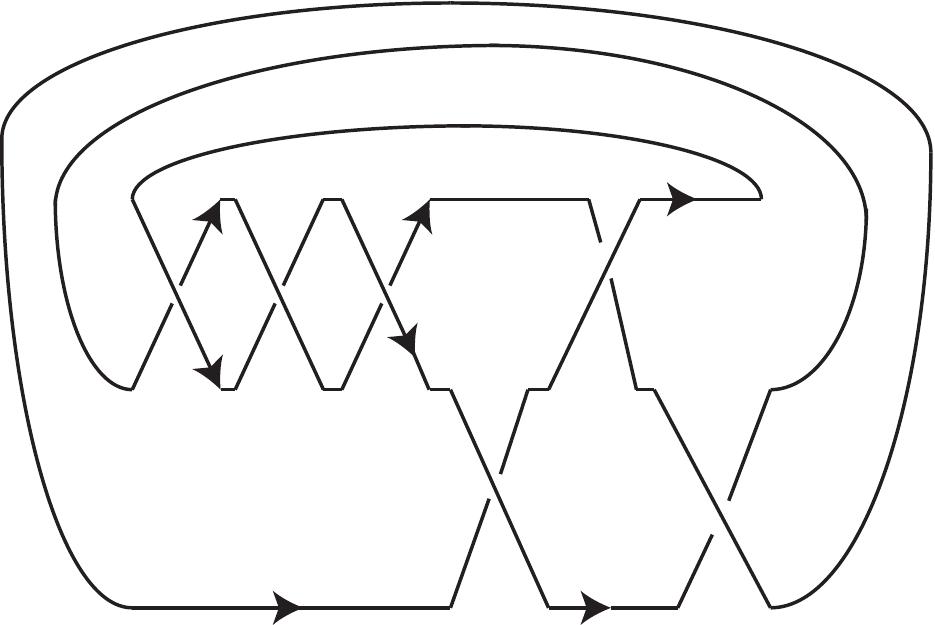}
      \caption{A braid representation of the knot $5_2$.}
      \label{F:Braid}
   \end{center}
\end{figure}

\medskip

\begin{question}
If we allow mosaics to be rectangular, what is the smallest rectangular board on which we can place a knot?
\end{question}

\medskip

\begin{definition}
Let $t(K)$ denote the tile number of a knot $K$; the fewest non-$T_0$ tiles needed to construct a given knot.
\end{definition}

\medskip

\begin{question}
What is the tile number of the knots of 10 or fewer crossings?
\end{question}

\medskip

\begin{question}
Is there an infinite family of knots whose tile number can be determined?
\end{question}

\medskip
Lastly, it should be noted that in the Appendix, the knots $8_3$, $8_6$,  $8_9$ and $8_{11}$ are depicted with nine or more crossing tiles.  Therefore, the crossing number of these knots are not realized in this representation.  

\medskip
\begin{question}
Does there exist a representation of $8_3$ (respectively $8_6$,  $8_9$ or $8_{11}$) on a $6\times 6$ board with only eight crossing tiles?
\end{question}
\medskip

These are just a few of the questions about knot mosaics that one can consider.  For those interested in studying knot mosaics using a tangible manipulative, visit Thingiverse.com to 3-D print your very own knot mosaics!

\vspace{8mm}


\section*{Appendix}

\label{A:AppendixA}
\begin{table}[ht]
\centering
\begin{tabular}{cccccc}
\includegraphics[width=.145\textwidth]{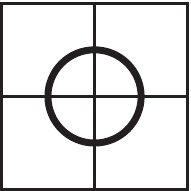}&
\includegraphics[width=.145\textwidth]{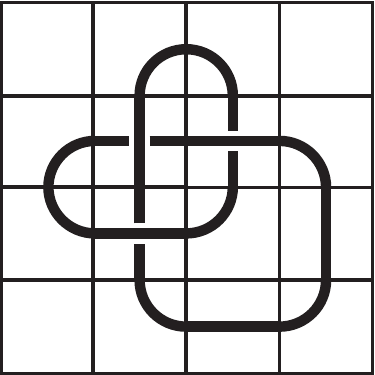}& 
\includegraphics[width=.145\textwidth]{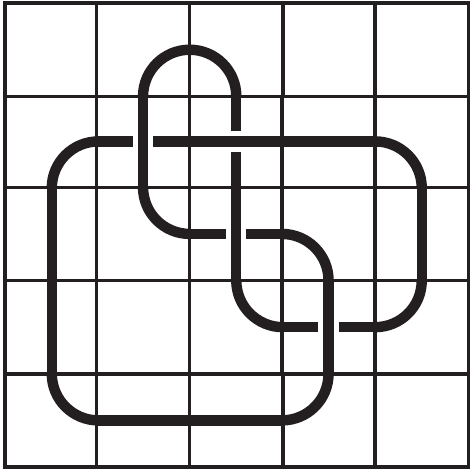}&
\includegraphics[width=.145\textwidth]{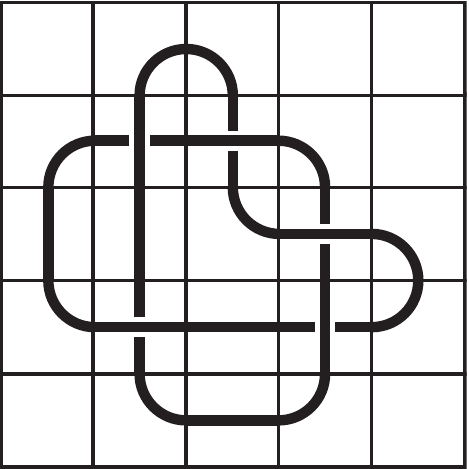}&
\includegraphics[width=.145\textwidth]{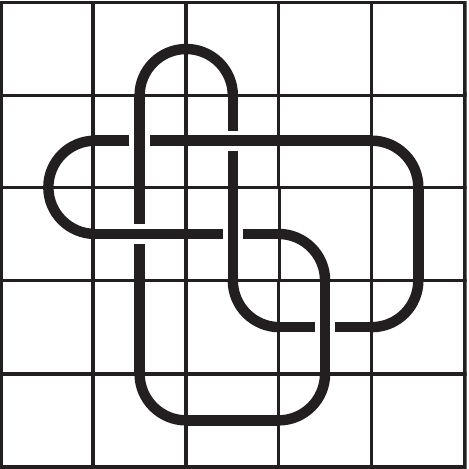}&
\includegraphics[width=.145\textwidth]{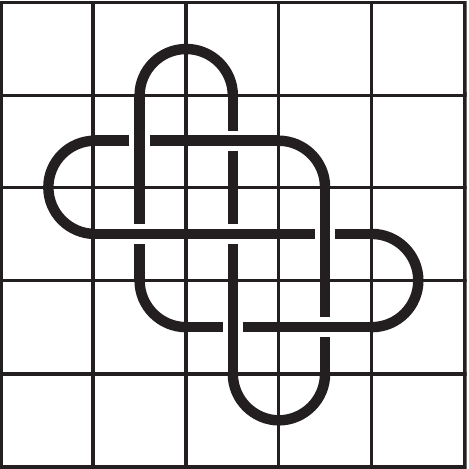} \\
$0_1$&$3_1$&$4_1$&$5_1$&$5_2$&$6_1$\\
&&&&&\\
\end{tabular}
\end{table}

\begin{table}
\centering
\begin{tabular}{cccccc}
\includegraphics[width=.145\textwidth]{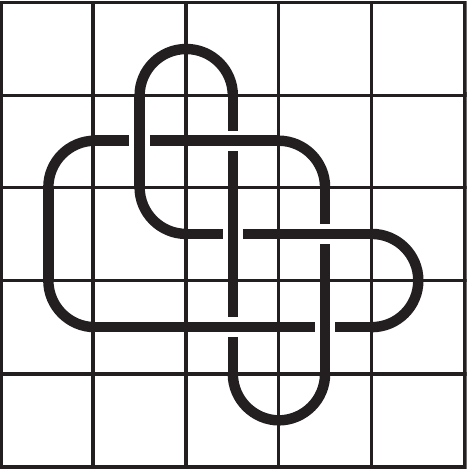}&
\includegraphics[width=.145\textwidth]{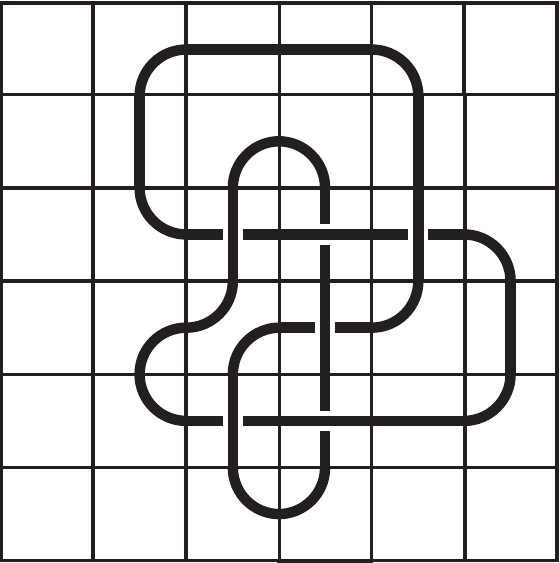}& 
\includegraphics[width=.145\textwidth]{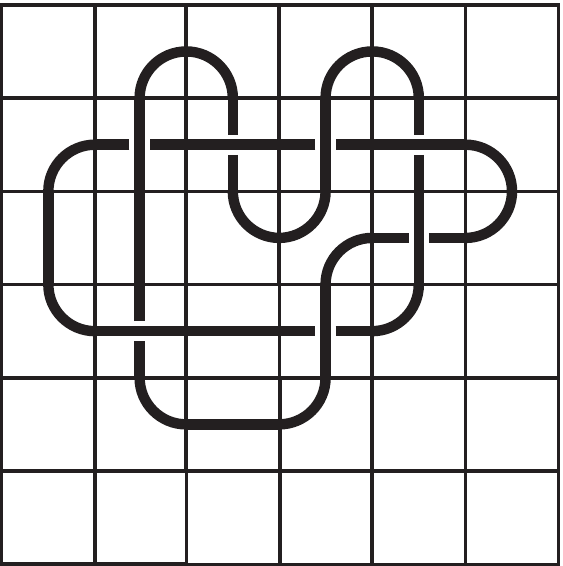}&
\includegraphics[width=.145\textwidth]{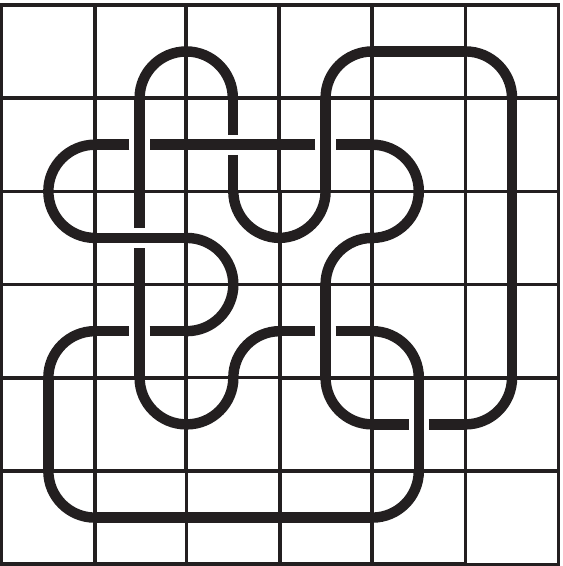}&
\includegraphics[width=.145\textwidth]{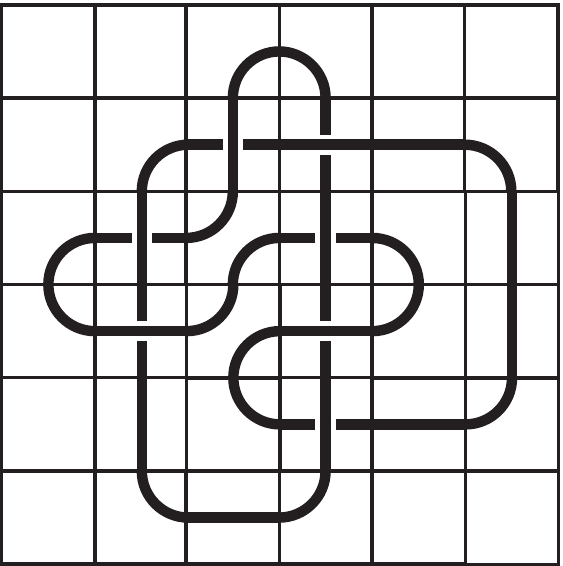}&
\includegraphics[width=.145\textwidth]{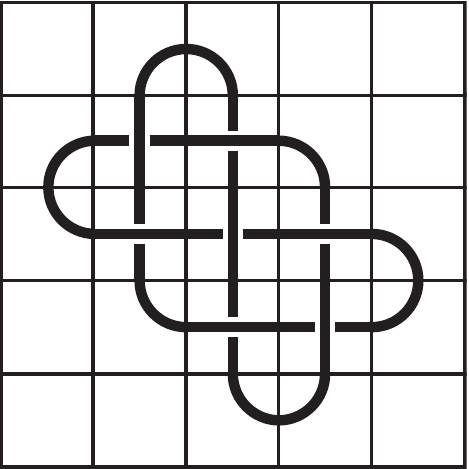}\\
$6_2$&$6_3$&$7_1$&$7_2$&$7_3$&$7_4$\\
&&&&&\\
\includegraphics[width=.145\textwidth]{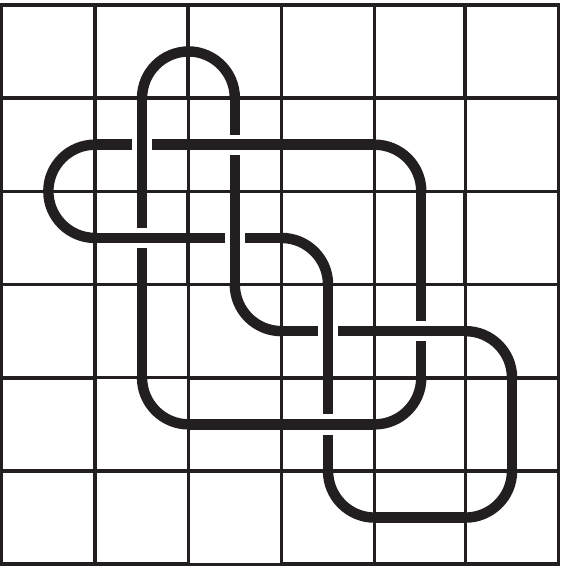}&
\includegraphics[width=.145\textwidth]{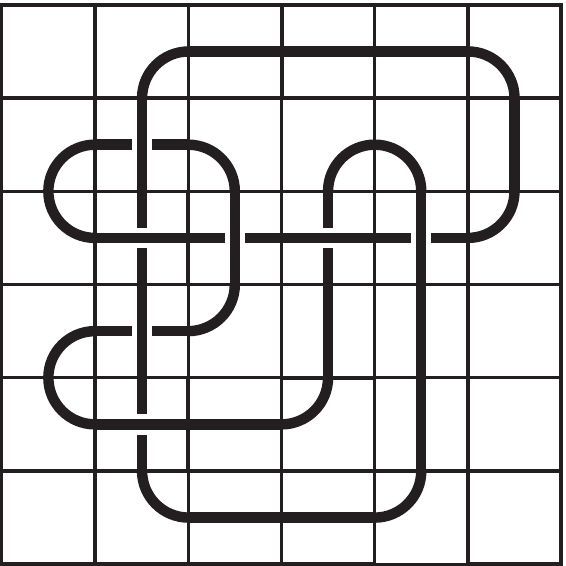}& 
\includegraphics[width=.145\textwidth]{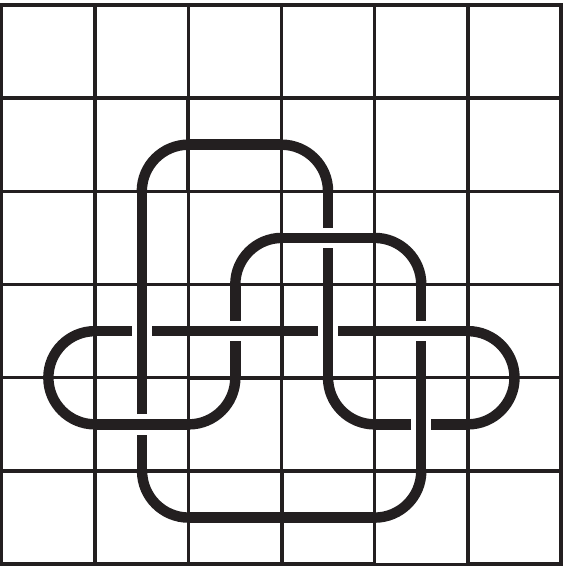}&
\includegraphics[width=.145\textwidth]{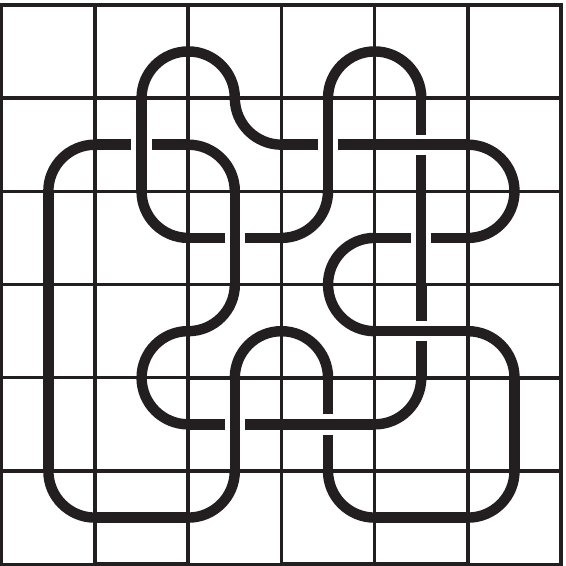}&
\includegraphics[width=.145\textwidth]{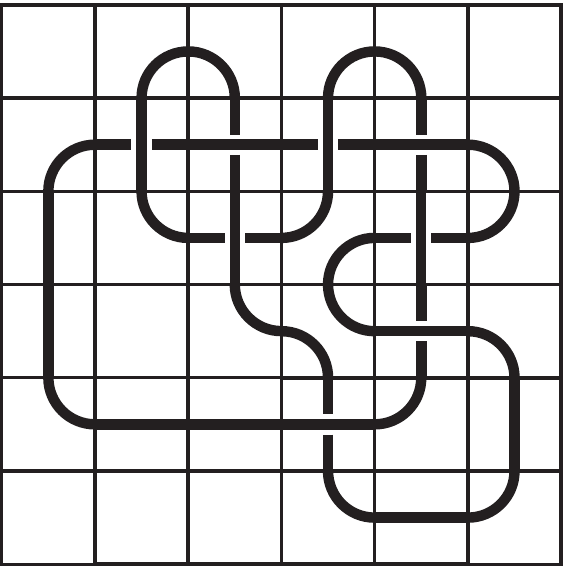}&
\includegraphics[width=.145\textwidth]{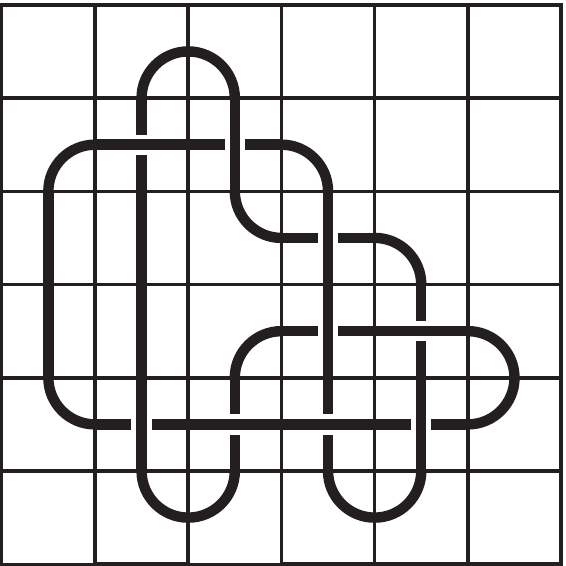}\\
$7_5$&$7_6$&$7_7$&$8_1$&$8_2$&$8_3$\\
&&&&&\\
\includegraphics[width=.145\textwidth]{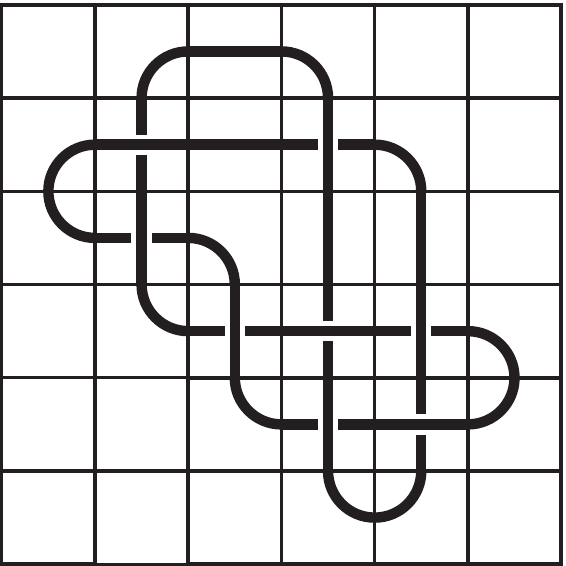}&
\includegraphics[width=.145\textwidth]{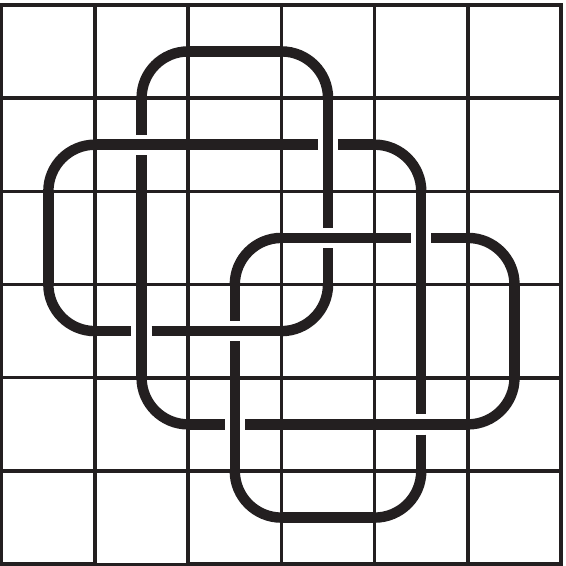}& 
\includegraphics[width=.145\textwidth]{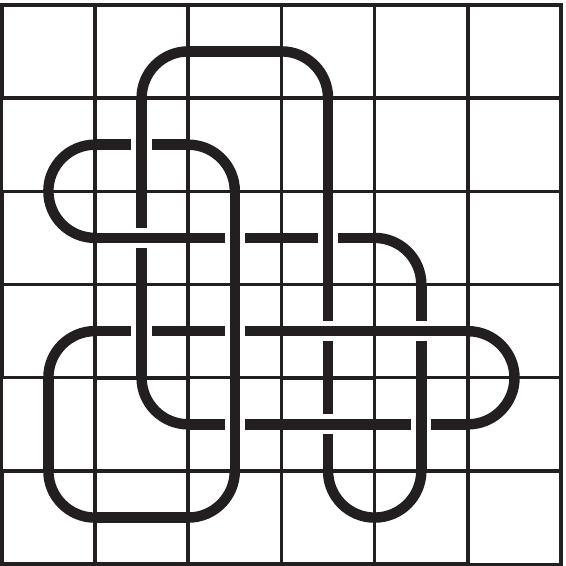}&
\includegraphics[width=.145\textwidth]{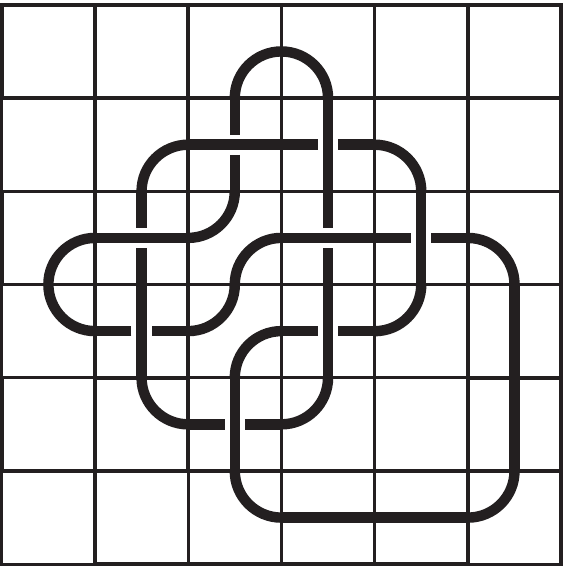}&
\includegraphics[width=.145\textwidth]{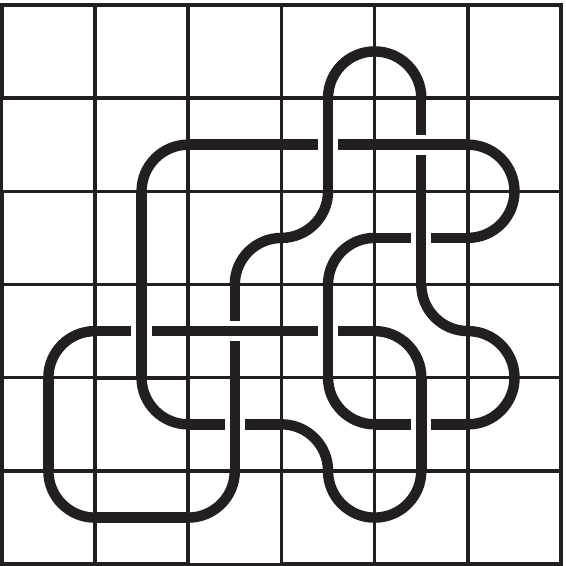}&
\includegraphics[width=.145\textwidth]{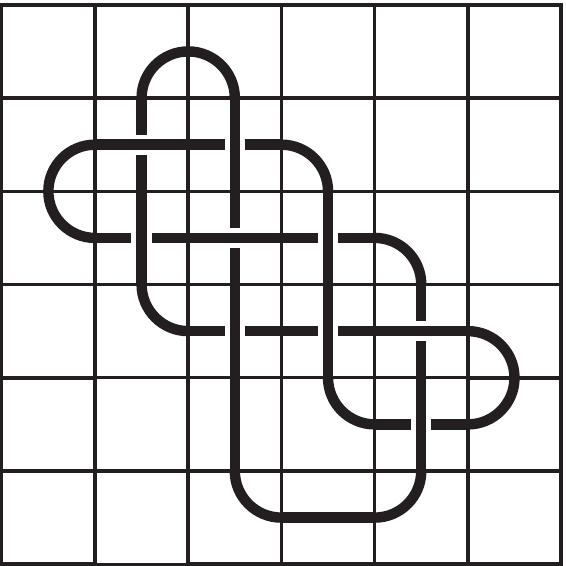}\\
$8_4$&$8_5$&$8_6$&$8_7$&$8_8$&$8_9$\\
&&&&&\\
\includegraphics[width=.145\textwidth]{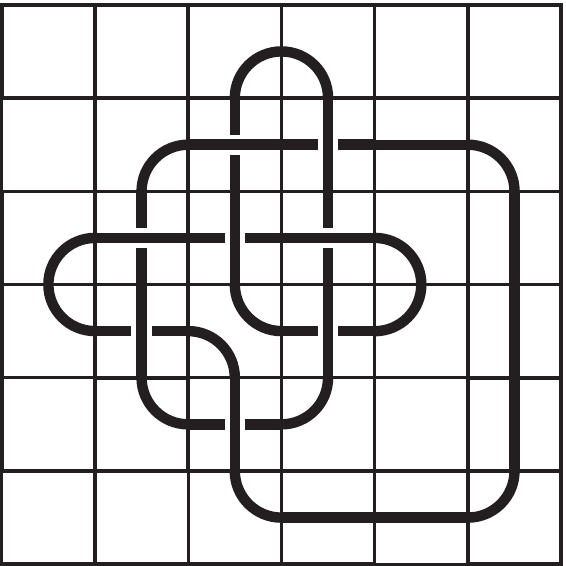}&
\includegraphics[width=.145\textwidth]{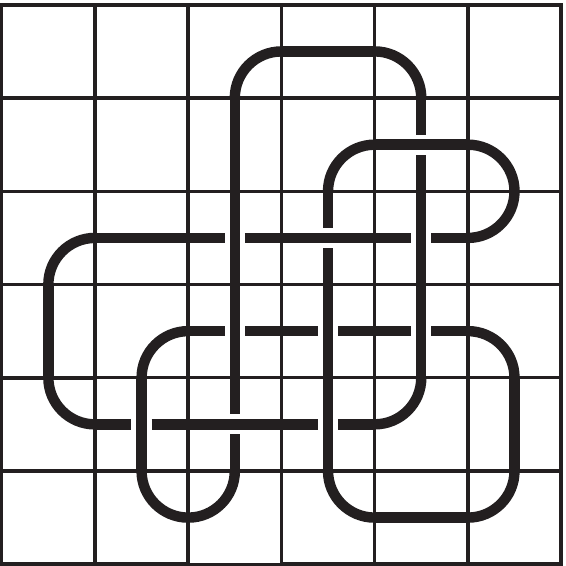}& 
\includegraphics[width=.145\textwidth]{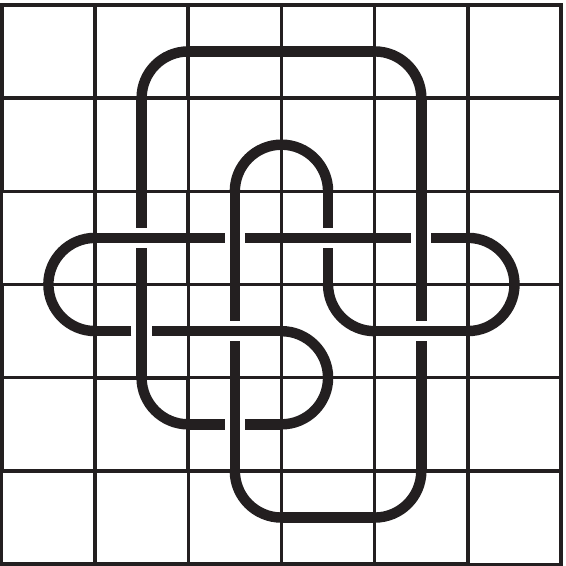}&
\includegraphics[width=.145\textwidth]{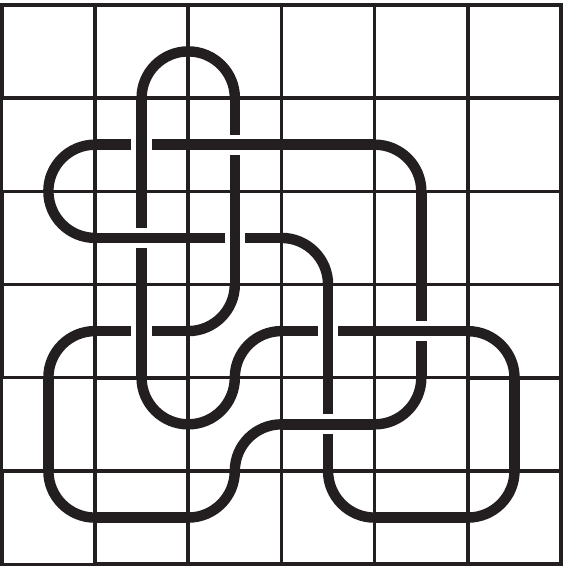}&
\includegraphics[width=.145\textwidth]{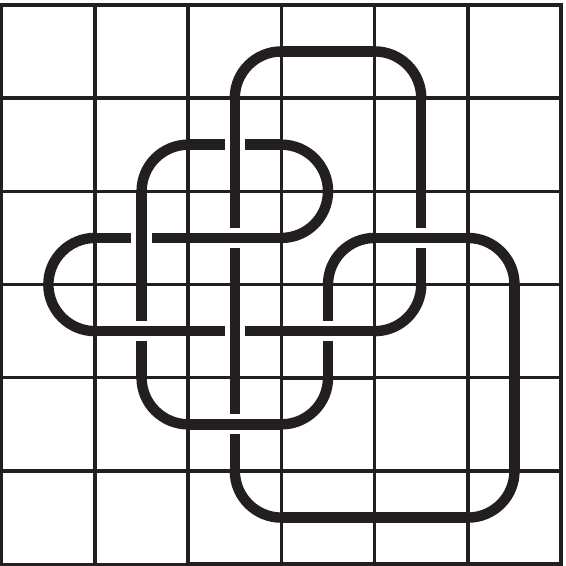}&
\includegraphics[width=.145\textwidth]{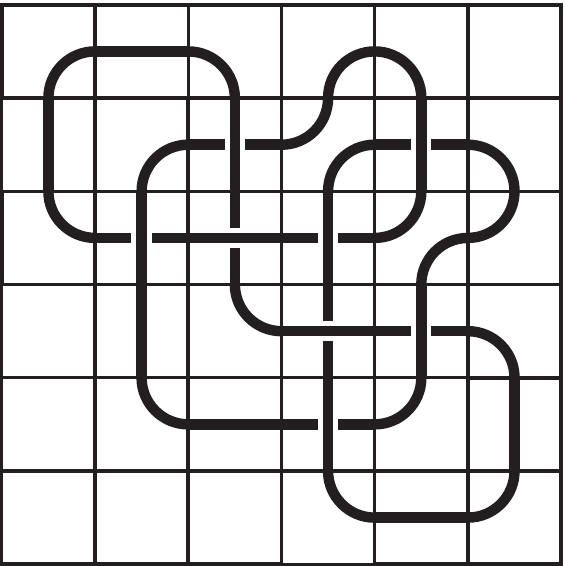}\\
$8_{10}$&$8_{11}$&$8_{12}$&$8_{13}$&$8_{14}$&$8_{15}$\\
&&&&&\\
\includegraphics[width=.145\textwidth]{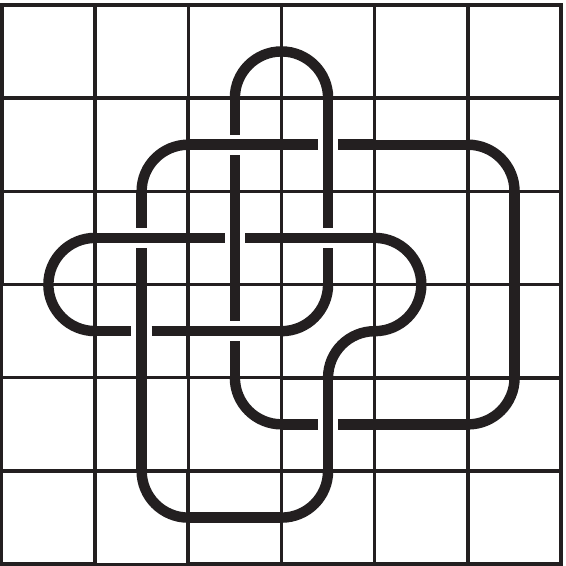}&
\includegraphics[width=.145\textwidth]{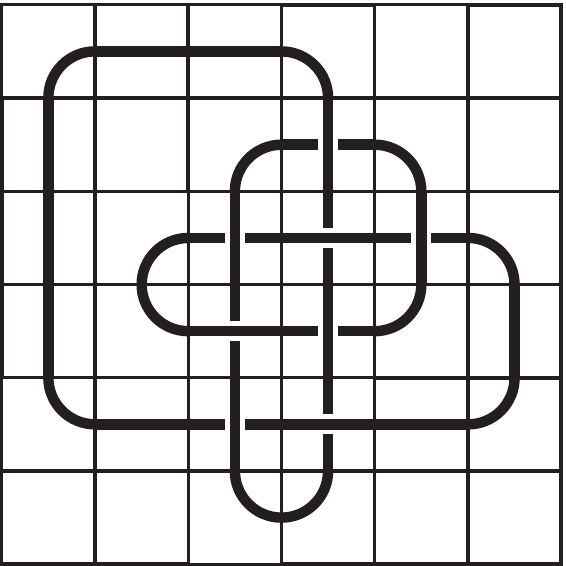}& 
\includegraphics[width=.145\textwidth]{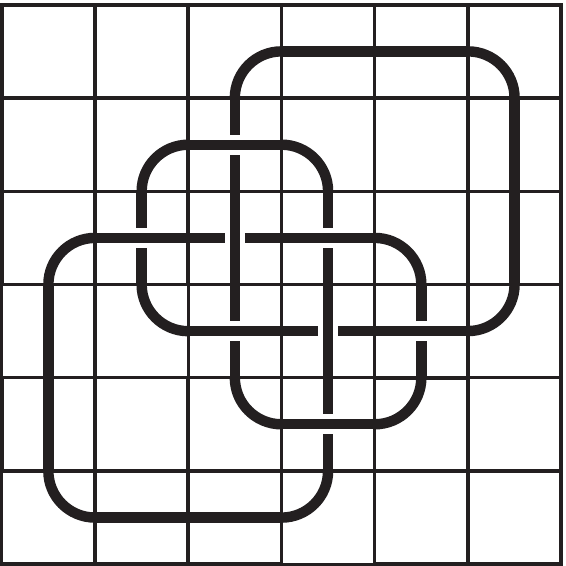}&
\includegraphics[width=.145\textwidth]{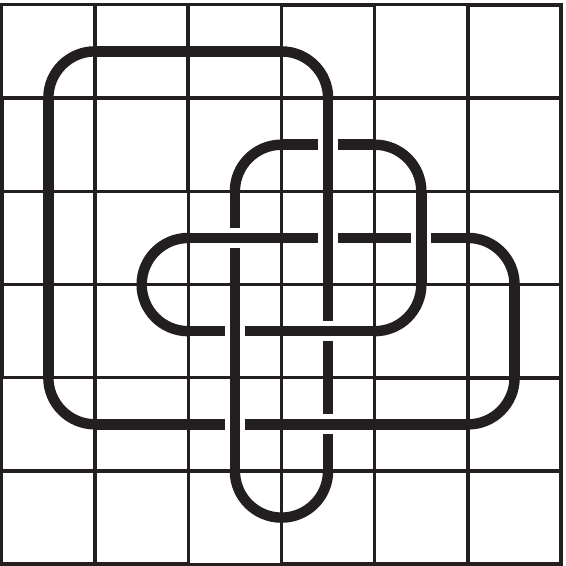}&
\includegraphics[width=.145\textwidth]{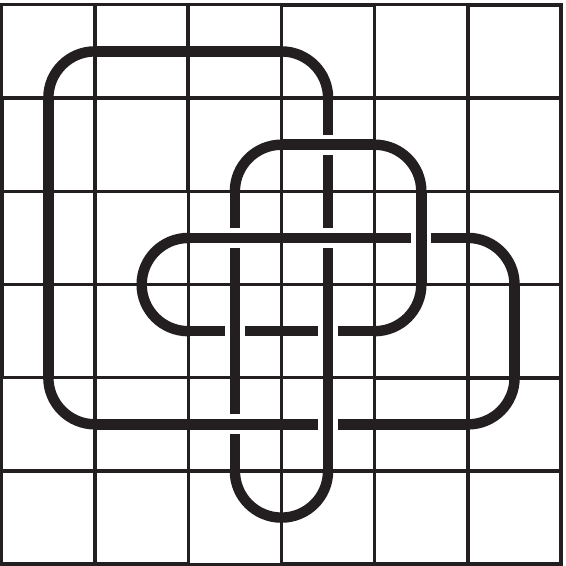}&
\includegraphics[width=.145\textwidth]{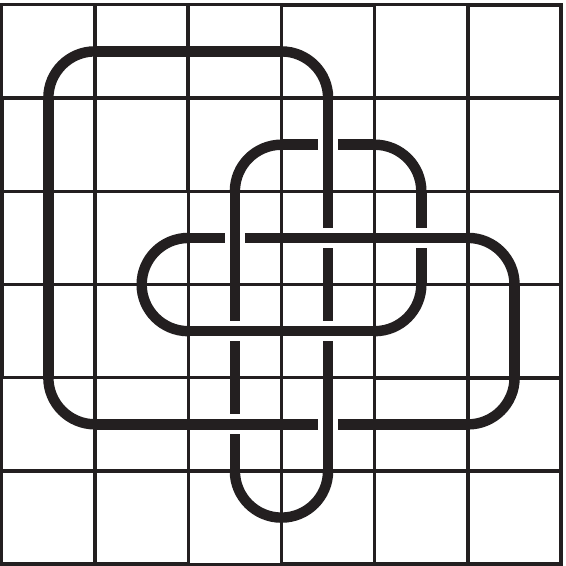}\\
$8_{16}$&$8_{17}$&$8_{18}$&$8_{19}$&$8_{20}$&$8_{21}$\\
&&&&&\\
\end{tabular}
\end{table}

\bibliographystyle{plain}

\end{document}